\documentclass{elsarticle}
\usepackage{amsmath, amsthm, amssymb}
\usepackage{graphicx}
\usepackage{cancel}
\usepackage{subcaption}
\usepackage{tikz}
\newcommand{\bs}{\backslash}
\newcommand{\mb}[1]{\mathbf{#1}}

\usetikzlibrary{shapes}
\newtheorem{definition}{Definition}
\newtheorem{example}{Example}
\newtheorem{proposition}{Proposition}
\newtheorem{condition}{Condition}
\newtheorem{property}{Property}

\begin{document}
	
\begin{frontmatter}
	\title{Formal Definitions of Conservative Probability Distribution Functions (PDFs)}
	
	\author[shane]{Shane Lubold}\ead{sl223@uw.edu}
	\author[clark]{Clark N. Taylor}\ead{clark.taylor@afit.edu}
	
	\address[shane]{Department of Statistics, University of Washington, Seattle, WA 98105}
	\address[clark]{Electrical and Computer Engineering, Air Force Institute of Technology, WPAFB, OH 45433}
		
	\begin{keyword}
		Distributed Data Fusion, Sensor Fusion, Distributed Estimation, Covariance Intersection
	\end{keyword}
		
\begin{abstract}
	Under ideal conditions, the probability density function (PDF) of a random variable, such as a sensor measurement, would be well known and amenable to computation and communication tasks. However, this is often not the case, so the user looks for some other PDF that approximates the true but intractable PDF. Conservativeness is a commonly sought property of this approximating PDF, especially in distributed or unstructured data systems where the data being fused may contain un-known correlations.  Roughly, a conservative approximation is one that overestimates the uncertainty of a system.  While prior work has introduced some definitions of conservativeness, these definitions either apply only to normal distributions or violate some of the intuitive appeal of (Gaussian) conservative definitions.  This work provides a general and intuitive definition of conservativeness that is applicable to any probability distribution that is a measure over $\mathbb{R}^m$ or an infinite subset thereof, including multi-modal and uniform distributions.  Unfortunately, we show that this \emph{strong} definition of conservative does not hold with any of the commonly used data fusion techniques.  Therefore, we also describe a weaker definition of conservative and show it is preserved through common data fusion methods, assuming the input distributions can be factored into independent and common PDFs that can be normalized over $\mathbb{R}^m$. By illustrating what is possible and not possible in terms of conservativeness during data fusion, an improved understanding of data fusion methods for general PDFs can be obtained.
\end{abstract}

\end{frontmatter}

\section{Introduction}
\label{s:intro}
	Ideally, the true probability density function (PDF) modeling a random event would be known, easily computed, and easily shared among cooperating agents. However, in many scenarios, such a PDF is not available, so alternative PDFs that have certain properties with respect to the original PDF are required instead. Consider the case of combining sensed information from multiple sources \cite{mutambara1998decentralized,da2021recent}. If we knew the correlation between the sources, then Bayesian data fusion would be applicable. However, if the correlation is unknown, it is generally accepted that we should overestimate the uncertainty on the output \cite{julier1997non}. In many cases, the joint model of information is simply unavailable or too costly or complex to determine\footnote{For example, when fusing information in a distributed system with numerous sensors, the correlation between all sensors must be maintained.  This quickly becomes impractical as the number of sensors grows.}. Thus, one often must take each ``piece'' of information ``marginally'' (or separately).  At the same time, one should be mindful of not taking each piece of information as independent of all others. One should therefore seek a \textit{conservative estimate}. But what does it mean to have a conservative estimate? To intuitively describe the idea of conservativeness, consider the following examples.
	
\begin{example}
\normalfont
\label{ex:Robot}
Consider a robot maneuvering through an area with an obstacle.  An on-board sensor is used to estimate the location of the obstacle, and a path planner is used to direct the robot away from the obstacle's location.  The path planner is designed to avoid a collision with a greater than 99.9\% accuracy, assuming an accurate PDF of the object's location is passed in.  A sensor on-board the robot detects the object and passes information (a PDF) about the estimated location to the path planner.  To ensure the obstacle is avoided, it is better for the generated PDF to be \emph{conservative} (i.e. the areas enclosed by the probabilistic bound, 99.9\% in our example, are larger than what the total information given by the sensor implies) rather than optimistic (the estimated 99.9\% probability area will be smaller than what the sensor measurements actually represent).
\end{example}

\begin{example}
    \normalfont
    \label{ex:distributed}
    Consider the multisensor, multitarget tracking problem and its solutions as described in \cite{li2019second,da2021recent}.  A distributed network of sensors is responsible for tracking multiple targets of interest.  This problem is typically addressed in one of two ways: using (1) data-level measurement fusion (where raw measurements are communicated between the sensors) or (2) estimate-level density fusion, where the best estimate of target locations' PDFs are communicated between the sensors.  This paper is focused more particularly on the second approach, estimate-level density fusion and we analyze some often-used data fusion techniques such as the linear (or arithmetic averaging) fusion and log-linear (geometric averaging) fusion techniques.  With estimate-level fusion, the goal is to generate a PDF that is ``conservative'' w.r.t. the data-level fusion technique.  While data-level fusion may be more accurate, it requires significantly more computation and communication.
\end{example}

While these examples help to intuitively define what a conservative PDF is, we desire a formal definition of conservativeness between two PDFs and in data fusion.  Specifically, we want to formally answer the following two questions representing different aspects of conservativeness. In what follows, let $p_1, \ p_2, \ p_f, \ p_t$ be PDFs.

\textbf{Question \#1:} Is $p_1$ conservative with respect to (w.r.t.) $p_2$?

\textbf{Question \#2:} Given a fusion rule $\mathcal{F}$ with inputs $p_1$ and $p_2$, is  $p_f=\mathcal{F}(p_1, p_2)$ conservative w.r.t. $p_t=\mathcal{B}(p_1,p_2)$, where $\mathcal{B}$ represents the ``optimal'' Bayesian fusion of PDFs with known common information?  (We define ``optimal'' Bayesian data fusion in \eqref{eq:optimal}.)

To formally answer these questions requires a formal definition of a conservativeness.  While there is a formal definition for Gaussian distributions, there is no commonly accepted definition for general distributions on $\mathbb{R}^m$ or an infinite subset thereof.  In this paper, we first introduce a definition, \emph{strictly conservative}, that expresses the idea of less certainty while requiring some notion of similarity between two PDFs.  This definition is, we believe, an intuitively correct definition for answering Question \#1 for PDFs on $\mathbb{R}^m$ or an infinite subset thereof.  Unfortunately, we show that this definition cannot be used to guarantee conservative data fusion, making it impractical for answering Question \#2.  We therefore introduce a weaker definition (\emph{weakly conservative}) that also expresses the idea of less certainty, but with a weaker notion of similarity.  We show that using this definition of conservative, we are able to answer Question \#2 in the affirmative for several previously introduced data fusion techniques.

\subsection{Prior definitions of conservative}
\label{ss:priorDefs}
For fusion of Gaussian PDFs, the positive semi-definite (p.s.d.) definition of conservativeness is commonly accepted \cite{AjglSimandl,Julier,noack2017decentralized,Noack,Niehsen,Ajgl}.  A matrix $\mb{A}$ is p.s.d. compared to $\mb{B}$ (denoted $\mb{A} \succeq \mb{B}$) if\footnote{With strict inequality, we have the positive definite ($\succ$) definition.}
\begin{equation}
x^T (\mb{A}-\mb{B})x \geq 0 \, \forall x.
\label{eq:psd}
\end{equation}
A Gaussian PDF that is the result of a data fusion algorithm $\mathcal{F}$ is p.s.d. conservative if its covariance (second order central moment) is greater than the true covariance of the underlying random variable ($E[xx^\top]$).  This definition has been used to prove that certain fusion rules are conservative (Question \#2).  For example, consider covariance intersection (log-linear polling), where it is shown in \cite{Niehsen} that if the input distributions are all conservative, then the output is also conservative.

Unfortunately, the p.s.d. definition of conservative does not easily extend to non-Gaussian distributions.  Previous work \cite{Julier,wang2012distributed} proposed extending the p.s.d. definition to non-Gaussian PDFs by comparing the second moment of the realized distribution with the true second moment after fusion. However, comparing specific moments between PDFs is not always informative. For example, consider the variances of two distributions: (1) a Gaussian distribution with large variance and (2) a multi-modal distribution consisting of the weighted sum of two Gaussian distributions with small variances but different means.  While the distribution with two Gaussians may have a larger variance (e.g., when the means of the two distributions are far apart), it is not clear that it should be conservative with respect to the single Gaussian distribution.

Another approach to defining conservativeness for general PDFs \cite{AjglSimandl,ajgl2014conservativeness,bailey2012conservative} requires that $H(p_c) \ge H(p_t)$, where $H$ is the differential entropy of a distribution. We call this the greater entropy (GE) condition.  Unfortunately, as with methods evaluating the second central moment, entropy by itself does not convey sufficient information to characterize an entire PDF.  Therefore, previous papers have combined the GE condition with other metrics to define conservativeness.

In \cite{AjglSimandl,ajgl2014conservativeness}, the increase in entropy from $p_t$ to $p_c$ is required to be greater than the Kullback-Leibler (KL) divergence between the two distributions.  Unfortunately, even this condition does not prevent some unsatisfactory distributions from being considered as ``conservative''.  One of the positive characteristics of p.s.d. conservative for Gaussians is that a marginalized covariance of the conservative distribution will be conservative w.r.t. the true distribution marginalized in the same direction.  The combined GE and KL do not ensure that this positive characteristic occurs.  This shortcoming in illustrated in Figure~\ref{fig:PSDELLIPSE} where Gaussian distributions are illustrated by a level set of that distribution, forming an ellipse.  While all p.s.d. conservative ellipses will entirely enclose the truth ellipse, the GEKL distribution does not.

\begin{figure}
	\minipage{0.38\textwidth}
	\scalebox{0.75}{
	\begin{tikzpicture}
		\draw[thick] (0,0) ellipse (2cm and 1cm);
		\draw[violet, thick, dotted] (0,0) ellipse (2cm and 1.5cm);
		\draw[blue, thick, dashdotted] (0,0) ellipse (3cm and 1.5cm);
		\draw[red, dashed] (0,0) ellipse (1.5cm and 3cm);
		\node[label=left:x1] at (-2cm,0) {\textbullet};
		\node[label=left:x2] at (0,1.25cm) {\textbullet};
	\end{tikzpicture}} 
	
	\centering (a)
	\endminipage
	\hspace{1mm}
	\minipage{0.61\textwidth}
	\scalebox{0.8}{
	\begin{tabular}{|p{.33\textwidth}|c|c|c|c|}
		\hline
		\centering Label & Covariance & p.s.d. & GEOP & GEKL\\\hline\hline
		Truth\par(black, solid) & $\begin{pmatrix}4 & 0\\0 & 1\end{pmatrix}$ & N/A & N/A & N/A \\\hline
		Candidate 1\par(violet, dotted) & $\begin{pmatrix}4 & 0 \\0 & 2.25\end{pmatrix}$ &  $\checkmark$ & X & $\checkmark$\\\hline
		Candidate 2\par(blue, dash-dot) & $\begin{pmatrix}9 & 0\\0 & 2.25\end{pmatrix}$ & $\checkmark$ & $\checkmark$ & $\checkmark$ \\\hline
		Candidate 3\par(red, dashed) & $\begin{pmatrix}2.25 & 0 \\0 & 9\end{pmatrix}$ & X & X & $\checkmark$\\\hline
	\end{tabular}} 
	
	\vspace{4mm}
	\centering (b)
	\endminipage
	\caption{In subfigure (a), level sets for the original (truth) PDF and three candidate conservative PDFs are shown.  Candidates 1 and 2 are conservative in the traditional, p.s.d. conservative sense, but only candidate 2 meets the order preserving (OP) condition.  Candidate 3 is not p.s.d. conservative, but still meets the GE and KL conditions from \cite{AjglSimandl,ajgl2014conservativeness}.  Subfigure (b) summarizes the graph in (a).}
	\label{fig:PSDELLIPSE}
\end{figure}
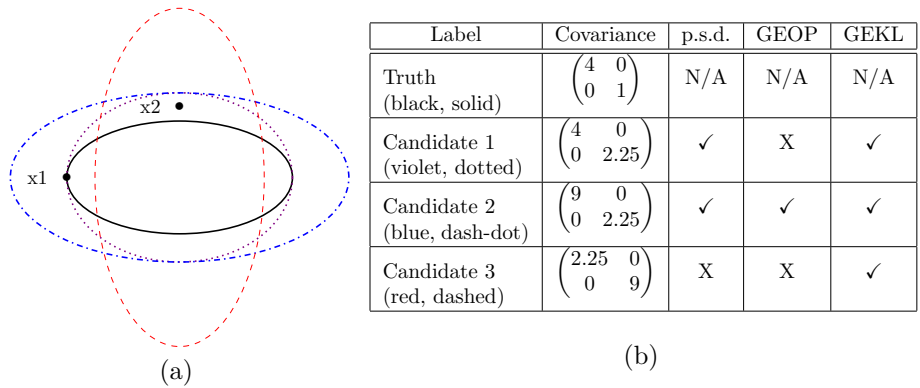

In \cite{bailey2012conservative}, an order preservation (OP) condition is added to the greater entropy condition. This condition states: for all $x_1, \ x_2$, $p_t(x_1) \ge p_t(x_2) \text{ iff } p_c (x_1) \ge p_c(x_2)$. Because of these two conditions, we refer to this paper's definition of conservative as \emph{GEOP conservative}. Unfortunately, the GEOP definition is overly strict, excluding many Gaussian distributions that could reasonably be considered conservative.  Consider the two points, $x_1$ and $x_2$ shown in Figure \ref{fig:PSDELLIPSE}.  Note that $x_1$ is on the level set ellipse for the \emph{truth} PDF, while $x_2$ is outside it.  This implies that $p_t(x_1) > p_t (x_2)$.  When the ellipse for \emph{Candidate 1} is considered,  $x_2$ is still approximately on the ellipse, while $x_1$ is significantly inside of it, showing that $p_c(x_1) < p_c(x_2)$, violating the order preservation condition.  For Gaussian distributions, the GEOP definition is significantly more restrictive than the p.s.d. definition (for more details see Section~\ref{s:ComparePrevious}).

The discussion of non-Gaussian distributions so far has primarily focused on Question \#1, but prior work has also attempted to define conservative data fusion for non-Gaussian PDFs (Question \#2). In \cite{julier2006using}, a definition of conservative data fusion was proposed stating ``an update rule is consistent if the probability of finding that the state is at $x$ is not reduced as a result of the update.''  While there is some intuitive reasoning for this definition, \cite{julier2006using} admits that this definition does not necessarily lead to useful fusion algorithms.  In \cite{bailey2012conservative} it states: ``A fusion rule is conservative if and only if it satisfies two properties: (1) It does not double count common information and (2) it replaces each component of independent information with a conservative approximation.''  While this definition of conservative data fusion was introduced in the same paper as the GEOP conservative definition, data fusion techniques that follow these two proposed rules do not produce PDFs that are GEOP conservative (see Section~\ref{ss:strictCons}).  This leads to the undesirable situation where a conservative fusion rule (Question \#2) produces PDFs that are not conservative w.r.t. the optimally fused distribution (as defined by the answer to Question \#1). 


\subsection{Contributions}
Given these shortcomings in related work, a new definition of conservative PDFs is desired.  This definition should be applicable to non-Gaussian PDFs (Question \#1), should enable the analysis of fusion rule outputs to answer question \#2, and for Gaussians should be roughly equivalent to the p.s.d. definition of conservative.  We would also like this definition to have some intuitive appeal.  

In this paper, we first propose a definition (strictly conservative) for answering Question \#1 for non-Gaussian PDFs. This definition is based off of the concept of minimum volume sets and is closely related to the \emph{matrix version} of p.s.d. for Gaussian distributions with the same mean.  We believe this definition has considerable intuitive appeal when answering Question \#1.  Unfortunately, we also show that in general this definition cannot hold through any fusion rule that does not know the common information between its inputs (Question \#2).  We therefore propose a weaker definition of conservative (also based off minimum volume sets) that is preserved through various data fusion methods assuming the input PDFs can be factored into common and independent PDFs \emph{and} all PDFs are normalizable (they integrate to a finite value across $\mathbb{R}^m$).  Given these assumptions, we can prove properties for data fusion techniques including the linear opinion pool \cite{clemen1999combining,Abbas2009}, Chernoff or log linear opinion pool \cite{Julier}, and homogeneous fusion \cite{Taylor2019}.  

The rest of the paper is as follows. We introduce our new definitions of conservativeness in section \ref{sec: Def_Cons}, and compare them against prior definitions in Section~\ref{s:ComparePrevious}.  In Section \ref{s:Applications} we show how our new definition of conservativeness can be used to verify the performance of ``conservative'' fusion rules, even for non-Gaussian distributions, assuming specific properties of the distributions being fused. In Section \ref{sec: Conclusion}, we provide concluding remarks.

\subsection{Notation}
For clarity, we briefly review our notation. We use $\mathrm{N}(\mu, \Sigma)$ to denote the PDF of a Gaussian random variable with mean $\mu$ and covariance $\Sigma$ and $\mathrm{U}(a, b)$ to denote the PDF of a continuous, uniform random variable on $(a, b)$. For a PDF $p$, we use $\text{supp}(p)$ to refer to the support of $p$.
For $\mathbf x = (x_1, \dotsc, x_n) \in \mathbb{R}^m$, we let $||\mathbf x||^2 = \sum_{i = 1}^n x_i^2$ be its squared $\ell^2$-norm. 
We denote the complement of a set $A$ as $A^\mathsf{c}$.  For probability distributions, we use lower case $p(x)$ to represent the PDF at $x$, while $P(A)$ returns the probability mass of $A$.

For fusion, we consider multiple (n) ``input'' distributions, denoted $p_i(x),\,i\in\{1\ldots n\}$ where
\begin{equation}
    p_i(x) := p(x|\mathbf{Z}_i) = \frac{p(\mathbf{Z}_i|x) p(x)}{\int p(\mathbf{Z}_i|x) p(x) dx}
	\label{eq:trueFullFusion}
\end{equation}
where $\mathbf{Z}_i$ is shorthand for the set of available `data' (or `measurements') defining the posterior $p_i(x):= p(x|\mathbf{Z}_i)$.  The definition of $p_i(x)$ is the optimal `Bayesian' posterior considering $\mathbf{Z}_i$ in isolation under the standing assumption that the data $\mathbf{Z}_i$ defining $p(\mathbf{Z}_i|\mathbf{x})$ is conditionally independent in and of itself.

If $p(\mathbf{Z}_i|\mathbf{x})$ and $p(\mathbf{Z}_j|\mathbf{x}),\ i\neq j$ are conditionally independent, then the optimal `Bayesian' data fusion result can be expressed as
\begin{equation}
	p_{opt}(\mathbf{x}|\mathbf{Z}_1 \ldots \mathbf{Z}_n) =  \frac{p(x) \prod_{i=1}^n p(\mathbf{Z}_i| x)}{\int p(x) \prod_{i=1}^n p(\mathbf{Z}_i| x) dx}
\end{equation}
The fusion problems we consider in this paper, however, assume there is some ``common'' information $\mathbf{Z}_C$ present in more than one set $\mathbf{Z}_i$.  We use the notation $\mathbf{Z}_{i\bs C}$ to denote the set of measurements in $\mathbf{Z}_i$ that is not in $\mathbf{Z}_C$.  $\mathbf{Z}_C$ is defined such that 
\begin{equation*}
    \mathbf{Z}_{i\bs C} \cap \mathbf{Z}_{j\bs C} = \varnothing, i\neq j.
\end{equation*}
For notational simplicity, we use the notation:
\begin{align}
    p_c(x) &:= \frac{p(\mathbf{Z}_C|x)p(x)}{\int p(\mathbf{Z}_C|x)p(x) dx}\\
	p_{i\bs C}(x) &:= \frac{p(\mb{Z}_{i\bs C} |x)}{\int p(\mb{Z}_{i\bs C}|x)\ dx}
\end{align}
Note that if $\mathbf{Z}_C = \varnothing$, then $p_C(x) = p(x)$, the prior in equation \ref{eq:trueFullFusion}.  Also note that this definition assumes all $p_{i\bs C}(x)$ distributions and the $p_C(x)$ distribution are normalizable, i.e. $\int p_{i\bs C}(x)\ dx$ and $\int p(\mathbf{Z}_C|x)p(x) dx$ are finite.

With this notation, we note that the ``optimal'' or ``true'' fusion result can be expressed as:
\begin{equation}
    p_t(x) \propto p_C(x) \prod_{i=1}^n p_{i\bs C}(x).
     \label{eq:optimal}
\end{equation}
We use this definition for ``optimal'' fusion throughout this paper.

\section{Definition of Conservativeness}
\label{sec: Def_Cons} 
This section has three subsections.  In the first subsection, we introduce the concept of the minimum volume sets of a PDF and discuss some of their properties. Using these results, in subsections \ref{ss:strictCons} and \ref{ss:WeaklyCons} we give two definitions of conservativeness:
strictly and weakly conservative.

\subsection{Minimum Volume Sets}
In this work, we only consider continuous random variables taking values in $\mathbb{R}^m$ with Borel $\sigma$-algebra $\mathcal{B}(\mathbb{R}^m)$.
Let $\lambda$ denote the Lebesgue measure on $\mathbb{R}^m$. 
We assume that all variables in this work have a $\lambda$-density. Recall that a probability measure $P$ has $\lambda$-density if there is a non-negative function $p$ such that for all $A \in \mathcal{B}(\mathbb{R}^m), P(A) = \int_A p(x)\ dx$.

\begin{definition}
Let $P$ be a probability measure with $\lambda$-density $p$. A minimum volume (MV) set for $P$ with area $\alpha \in (0, 1)$ is
\begin{equation}
M_p(\alpha) = \arg \underset{X \subseteq \mathbb{R}^m}\inf \left\{\lambda(X): P(X) \geq \alpha \right\}.
\end{equation}
\end{definition}

Informally, a MV set is a set with the smallest volume that has probability at least $\alpha$. We now give three examples of MV sets.

\begin{example}
If $p = \mathrm{N}(0,1)$, the MV set for $p$ with area $\alpha$ is 
\begin{equation*}
    M_p(\alpha) = \left\{x : |x| \leq \Phi^{-1}\left( \frac{\alpha + 1}{2}\right)\right\} \;,
\end{equation*}
where $\Phi$ is the CDF of the standard Gaussian. 
\end{example}

\begin{example}
    Consider $p = \text{Exp}(\lambda)$, with PDF $p(x) = \lambda \exp(-\lambda x).$ The (unique) MVS for $p$ with area $\alpha$ is 
    \begin{equation*}
        M_p(\alpha) = \left[0, -\frac{\log(1 -\alpha)}{\lambda}\right] \;.
    \end{equation*}
Note that $\log(1-\alpha) < 0$ for $\alpha \in (0, 1)$ so the MV set is a subset of $[0, \infty).$
\end{example}

\begin{example}
Let $p = \frac{1}{3}\mathrm{N}((2, 4), \Sigma_1) + \frac{2}{3}\mathrm{N}((-1, -3), \Sigma_2)$ where 
\begin{equation*}
\Sigma_1 = 
\begin{pmatrix}
6 & 2 \\
2 & 3
\end{pmatrix} \;, \ \ 
\Sigma_2 = 
\begin{pmatrix}
5 & -1 \\
-1 & 4
\end{pmatrix} \;.
\end{equation*}
In Figure \ref{fig: MVS_Mixture}, we plot the boundaries of MV sets for four values of $\alpha.$ The smallest value of $\alpha$ corresponds to the green set, and the values of $\alpha$ then increase as we go from the sets outlined in black, red, and blue.
\label{ex: MVS_Mixture}
\end{example}

To understand the intuitive appeal of using MV sets to define conservativeness, we first discuss some properties of these sets.  See \cite{garcia2003level} for proofs of these properties.

\begin{property}
Every MV set is associated with a \emph{super-level set} $S_p(\beta):=\{x : p(x) \ge \beta \}$. Under some regularity conditions on $p$, the boundary of each $S_p(\beta)$ is a level set of the form $L_p(\beta) := \{x: p(x) = \beta\}$.  
\label{prop: LevelSet_MV}
\end{property}

MV sets are not always unique. If $p$ has a level set of non-zero Lebesgue measure (i.e. $p$ has a ``flat'' region), there exists an $\alpha$ for which there are multiple minimum volume sets. For example, if $p = \mathrm{U}(0, 1)$, the sets $(t, t + \alpha)$ for $t \in [0, 1 - \alpha]$ are all MV sets with area $\alpha$. We call a PDF \emph{non-flat} if all its level-sets have zero measure.

\begin{property}\label{prop:alphaSubSets}
If $\alpha_1, \alpha_2 \in [0,1]$ with $\alpha_1 > \alpha_2$, there exists $M_p(\alpha_1)$ such that $M_p(\alpha_1) \supset M_p(\alpha_2)$. If $p$ is non-flat, the set $M_p(\alpha_1)$ is unique.
\label{prop: Increasing_MV}
\end{property} 

\begin{property}\label{prop:greaterInside}
For $x \in M_p(\alpha)$ and $y \in M^\mathsf{c}_p(\alpha), \ p(x) \geq p(y).$ In addition, if $p$ is non-flat, then $p(x) > p(y)$.
\end{property}

From Properties \ref{prop:alphaSubSets} and \ref{prop:greaterInside}, we can associate to a PDF $p$ a collection of MV sets. Clearly, $M_p(1) = \text{supp}(p)$, while $M_p(\alpha)$ ``tightens'' around the higher likelihood areas of $p$ as $\alpha \rightarrow 0$. The boundaries of the MV sets create a ``topographical'' representation of $p$ (see Figure \ref{fig: MVS_Mixture}).

\begin{figure}
\centering\includegraphics[scale = 0.28]{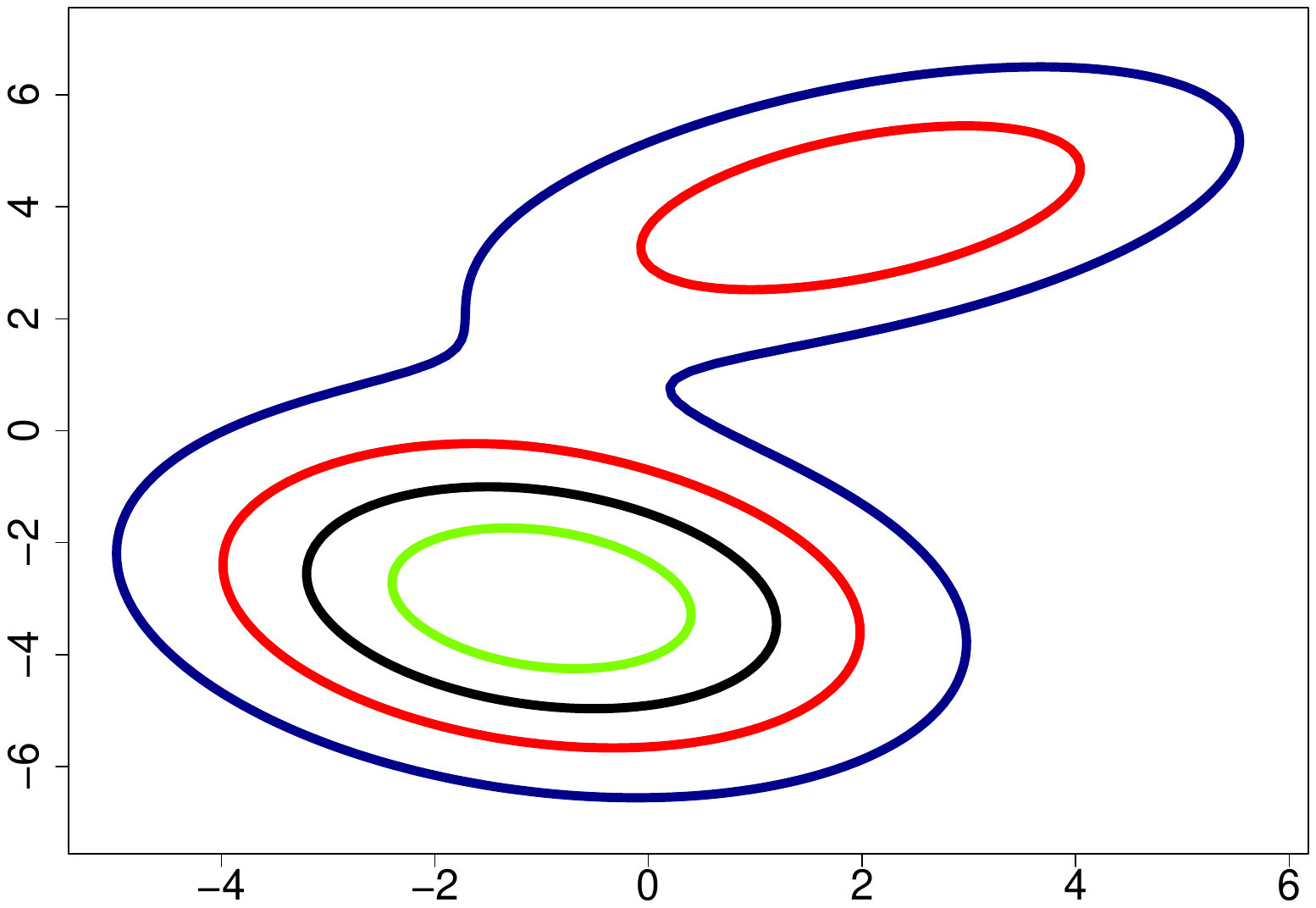}
\caption{Examples of  MV sets for the mixture of Gaussians in Example \ref{ex: MVS_Mixture}.
The values of alpha decrease as we go from the sets outlined in blue, red, black, and green. }
\label{fig: MVS_Mixture}
\end{figure}

\subsection{Strictly Conservative Definition}
\label{ss:strictCons}
We now state our proposed definition for \emph{strictly conservative}.
\begin{definition}
The PDF $p_c$ is \emph{strictly conservative} w.r.t. $p_t$ if for all $\alpha \in [0, 1]$, for each $M_t(\alpha)$, there exists a $M_c(\alpha)$ such that $M_c(\alpha) \supseteq M_t(\alpha).$
\label{def: Original_Cons}
\end{definition}

Informally, consider when $p_t$ and $p_c$ are 2-dimensional, non-flat PDFs.  The strictly conservative definition states that if the topographical lines (level sets) for both $p_t$ and $p_c$ were drawn on the same map, the topographical lines from $p_c$ would always enclose the corresponding lines from $p_t$.  Note that previous papers discussing conservativeness for Gaussian PDFs have often used ellipses representing a particular level set to illustrate the concept of conservativeness.  In many ways, the strictly conservative definition is a formalization and extension of the pedagogical diagrams frequently included in previous papers.

The strictly conservative definition has the following appealing attributes. First, this definition can be applied to any PDF, not just Gaussians. Second, this definition captures the intuition in Example~\ref{ex:Robot}, since for all $\alpha$ (the probability of a region including an object), the conservative distribution's region
is a super set of the true distribution's region.  Third, for two Gaussian distributions with the same mean, the p.s.d. and strictly conservative definitions are equivalent. Consider the following examples demonstrating some of the appealing attributes just described: 
\begin{example}
	Let $p_c=U(a,b)$ and $p_t=U(c,d)$. Then $p_c$ is strictly conservative w.r.t. $p_t$ if $(a,b) \supseteq (c,d)$.
\end{example}

\begin{example}
	Let $p_c$ be a Student's-t distribution with $\nu > 0$ degrees of freedom and $p_t = \mathcal{N}(0,1)$.  Then $p_c$ is strictly conservative w.r.t. $p_t$ for all $\nu$.
\end{example}

Despite the intuitive appeal of this definition, we cannot apply this definition to data fusion for general PDFs.  To understand this weakness of the strictly conservative definition, we first define \emph{maximum likelihood mode(s)} for non-flat PDFs.

\begin{definition}
For any non-flat PDF, there are a finite set of points, the maximum likelihood modes $\mathcal{M}_1 = \{x : p_1(x) = \sup p_1\}$ where the PDF reaches it maximum value. An alternate definition is $\mathcal{M}_1 = \lim_{\alpha \rightarrow 0^+} M_1(\alpha)$.
\end{definition}

\begin{property}
For $p_c$ to be strictly conservative w.r.t. $p_t$, it must be true that $\mathcal{M}_t \subseteq \mathcal{M}_c$.  
\end{property}
\begin{proof} If this were not true, then for small enough $\alpha$, there would exist an $M_t(\alpha)$ such that $M_t(\alpha) \not \subseteq M_c(\alpha)$, which is a contradiction.\end{proof} 

Now consider the data fusion problem: we desire a fusion rule $\mathcal{F}$ for fusing two PDFs, $p_1$ and $p_2$, such that $p_f \propto \mathcal{F}(p_1, p_2)$.  Furthermore, assume we can write the input information as
\begin{equation*}
p_1(x) \propto p_{1 \setminus C}(x) p_C(x), \ \ \ p_2(x) \propto p_{2 \setminus C}(x) p_C(x) \;,
\end{equation*}
where $p_C$ is the common information in $p_1$ and $p_2$. The goal of $\mathcal{F}$ is to generate a conservative approximation of $p_t \propto p_{1 \setminus C}p_{2\setminus C}p_C$ even when $p_C$ is not known.  

If $p_{1\setminus C}$, $p_{2\setminus C}$, and $p_C$ are Gaussian distributions with different means, then the sets $\mathcal{M}_1,\mathcal{M}_2$ and $\mathcal{M}_t$ are all single points located at the mean of the distribution.  Unfortunately, determining $\mathcal{M}_t$ is not possible from $p_1$ and $p_2$ alone since it requires knowledge of $p_C$.  Because $\mathcal{M}_f \supseteq \mathcal{M}_t$ must be true for $p_f$ to be strictly conservative w.r.t. $p_t$, creating a fusion rule that generates strictly conservative PDFs is impossible, even for the simple Gaussian case.  Therefore, in the following section we introduce a weaker definition of conservative that does not require that $\mathcal{M}_c = \mathcal{M}_t$ but that preserves some of the positive attributes of the strictly conservative definition.

\subsection{Weakly Conservative: A Definition of Conservativeness for Data Fusion}
\label{ss:WeaklyCons}

To create a definition of $p_c$ being conservative w.r.t. $p_t$ that can be used to evaluate fusion rules and is easy to verify, we introduce three conditions that capture much of the intuition behind Definition~\ref{def: Original_Cons}. 

\begin{condition}
$supp(p_c)\supseteq supp(p_t)$.
\label{cond:SupportSuperset}
\end{condition}
We say that conditions \ref{cond:Broader_Cons2} and \ref{cond:Broader_Cons3} hold for a given $\alpha$ if
\begin{condition}
$P_t(M_t(\alpha)) \geq P_c(M_t(\alpha))$, and
\label{cond:Broader_Cons2}
\end{condition}
\begin{condition}
There exists $M_c(\alpha) \ s.t. \ P_t(M_c(\alpha)) \geq P_c(M_c(\alpha)).$
\label{cond:Broader_Cons3}
\end{condition}
In the next Proposition, we connect these three conditions to the definition of strictly conservative. We prove this Proposition in the Appendix.

\begin{proposition}
If Conditions \ref{cond:SupportSuperset}, \ref{cond:Broader_Cons2}, and \ref{cond:Broader_Cons3} hold for all $\alpha \in (0, 1]$, this is necessary, but not sufficient, for $p_c$ to be strictly conservative w.r.t. $p_t$.
\label{prop: ThreeCond_StrictCons}
\end{proposition}

We now motivate these conditions. Condition~\ref{cond:Broader_Cons2} says that $P_c$ should assign less probability than $P_t$ to the MV sets of $p_t$.  This is similar to the motivation for conservativeness given in Example \ref{ex:Robot}. Condition~\ref{cond:Broader_Cons3} requires that areas of high probability for $p_c$ are also areas of high probability for $p_t$.  

\begin{definition}
A PDF $p_c$ is \emph{weakly conservative} w.r.t. $p_t$ if Condition~\ref{cond:SupportSuperset} is met and Conditions~\ref{cond:Broader_Cons2} and \ref{cond:Broader_Cons3} hold for all $\alpha \in [\alpha',1]$, where $0 \leq \alpha' < 1$.
\label{def: Broader_Cons}
\end{definition}
We note that as $\alpha'$ decreases, the ``similarity'' of the two PDFs should also increase.  Consider the following four examples that illustrate when a PDF is and is not weakly conservative. In later sections we discuss how this definition compares with other definitions and how it can be used to evaluate fusion rules.

\begin{figure*}
	\centering
	\minipage{0.32\textwidth}
	\includegraphics[width=1\linewidth]{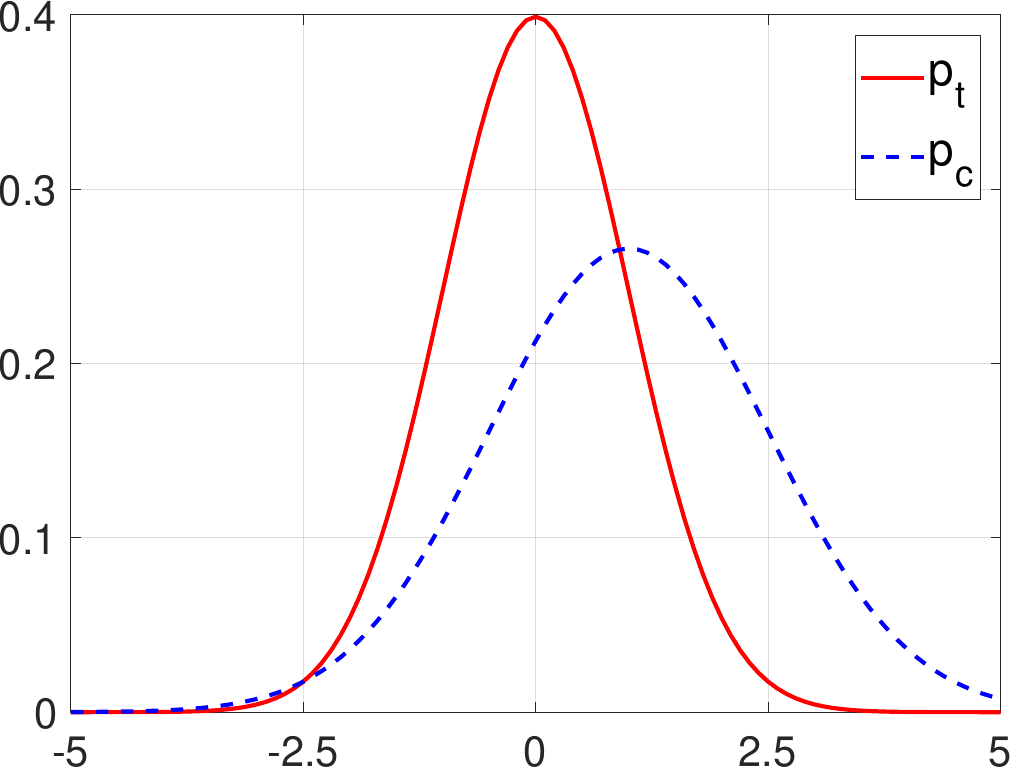}
	\centering (a)
	\endminipage
	\hspace{1mm}
	\minipage{0.32\textwidth}
	\includegraphics[width=1\linewidth]{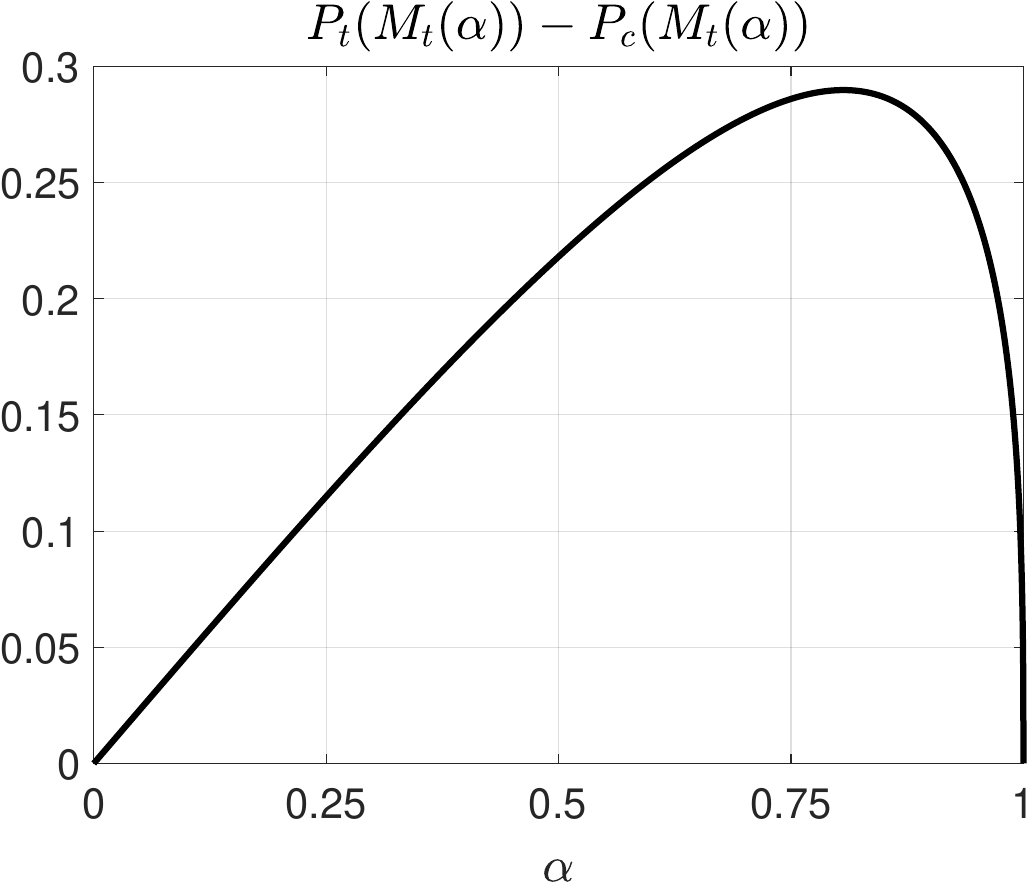}
	\centering (b)
	\endminipage
	\hspace{1mm}
	\minipage{0.32\textwidth}%
	\includegraphics[width=1\linewidth]{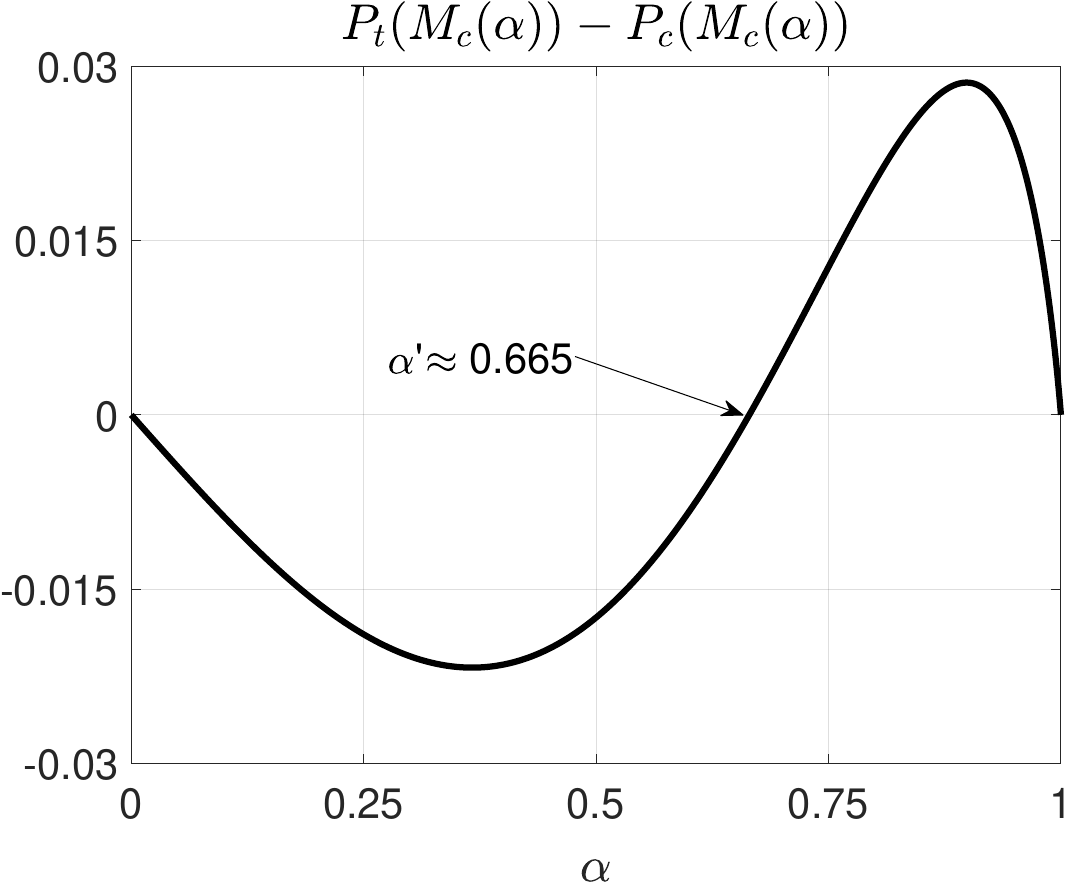}
	\centering (c)
	
	\endminipage
	\caption{On the left we plot $p_t = \mathrm{N}(0, 1)$ and $p_c = \mathrm{N}(1, 2.25)$. $p_c$ is weakly, but not strictly, conservative w.r.t $p_t$. To illustrate this point, in the middle figure above we plot $P_t(M_t(\alpha)) - P_c(M_t(\alpha))$ for $\alpha \in [0, 1]$. On the right, we plot $P_t(M_c(\alpha)) - P_c(M_c(\alpha))$ for $\alpha \in [0, 1].$ Note that Condition \ref{cond:Broader_Cons2} is satisfied for all $\alpha$, while Condition \ref{cond:Broader_Cons3} is only satisfied for $\alpha > \alpha'$, where $\alpha' \approx 0.65.$}
	\label{fig:weakly_example}
\end{figure*}


\begin{example}
Let $p_1=\mathrm{N}(\mu_1,\Sigma_1)$ and $p_2=\mathrm{N}(\mu_2,\Sigma_2)$. Then $p_1$ is weakly, but not strictly conservative w.r.t. $p_2$ if $\mu_1 \neq \mu_2$ and $\Sigma_1 \succ \Sigma_2$. Consider Figure \ref{fig:weakly_example}, where $p_t = \mathrm{N}(0,1)$ and $p_c=\mathrm{N}(1,1.5^2)$.  Because both distributions are Gaussian, Condition 1 is met.  As subfigures \ref{fig:weakly_example}(b)  and \ref{fig:weakly_example}(c) show, conditions 2 and 3 are also met for $\alpha \in [\alpha', 1]$, with $\alpha' \approx 0.665.$
\label{ex:nonEqualMeans}
\end{example}

\begin{example}
Consider two mixtures of Gaussians $p_1$ and $p_2$,
\begin{equation}
p_1 = \sum_{i = 1}^n \omega_i \mathrm{N}(\mu_i,\Sigma_i), \ \ p_2 = \sum_{i = 1}^m \epsilon_i\mathrm{N}(\nu_i, \Phi_i) \;,
\label{eq: GaussianMix}
\end{equation}
where $\sum_{i = 1}^n \omega_i = \sum_{i = 1}^m \epsilon_i = 1$,  $\mu_i, \ \nu_i \in \mathbb{R}^k$ and $\Sigma_i, \Phi_i \in \mathbb{R}^{k \times k}.$ Assume there exists an index $\ell$ such that $\Phi_\ell \succeq \Phi_i$ for $i = 1, \dotsc, m$ and an index $k$ such that $\Sigma_k \succeq \Sigma_i$ for $i = 1, \dotsc, n$. If $\Sigma_k \succ \Phi_\ell$ then $p_1$ is weakly conservative w.r.t. $p_2$.
\label{ex: MoG}
\end{example}

\begin{example}
Let $p_1 = \mathrm{N}(0, 1)$ and $p_2 = \mathrm{N}(1, 1)$. Then $p_1$ is \emph{not} weakly conservative w.r.t. $p_2$.
\end{example}
 
\begin{example}
    Let $p_1$ and $p_2$ be skew-normal distributions, defined as $p(x) = \frac{2}{\sqrt{2\pi}}e^{-x^2/2}\Phi(sx)$, where $\Phi(y)$ is the CDF of the normal distribution.  If $s_1 \neq 0$ and $s_1 = -s_2$, then neither distribution will be weakly conservative w.r.t. the other.
\end{example}

\section{Comparison with Previous Definitions}
\label{s:ComparePrevious}
In this section, we describe the relationship between the strictly and weakly conservative definitions with previous definitions of conservativeness. In the first subsection, we focus on Gaussian distributions, followed by a discussion on other distributions.

\subsection{Gaussian PDFs}  
When determining whether $p_c = \mathrm{N}(\mu_c,\Sigma_c)$ is conservative w.r.t. $p_t = \mathrm{N}(\mu_t,\Sigma_t)$ (Question \#1 in the introduction), there are several definitions of conservativeness that one can use.  These definitions include the definitions discussed in Section \ref{ss:priorDefs} -- the p.s.d definition that apply specifically to Gaussians, and the the GEOP and GEKL definitions for general distributions -- and the two definitions proposed in this work: strictly conservative (SC -- Definition \ref{def: Original_Cons}) and weakly conservative (WC -- Definition \ref{def: Broader_Cons}).

In Table~\ref{tab:gaussCompare}, we show when $p_c$ is conservative w.r.t $p_t$ for each definition.  Each row represents various relationships between $p_c$ and $p_t$, and each column represents a different definition of conservative.  In general, the columns are ordered from most restrictive to least restrictive. In \ref{a:tableProofs}, we provide justification for several entries in Table~\ref{tab:gaussCompare}.

\begin{table*}[t]
	\begin{centering}
		\scalebox{0.9}{
		\begin{tabular}{lr|c|c|c|c|c}
			&&\textbf{GEOP} & \textbf{SC} & \textbf{WC} & \textbf{p.s.d.} & \textbf{GEKL}\\\hline\hline
			&\textbf{$\Sigma_c = k\Sigma_t, k \geq 1$} & $\checkmark$ & $\checkmark$ & $\checkmark$ & $\checkmark$ & $\checkmark$\\\cline{2-7}
			\textbf{$\mu_c = \mu_t$} & \textbf{$\Sigma_c \succeq \Sigma_t$} & X & $\checkmark$ & $\checkmark$ & $\checkmark$& $\checkmark$\\\hline
	        $\mathbf{d}=\mu_c-\mu_t$ &\textbf{$\Sigma_c \succeq \Sigma_t + \mathbf{dd}^\top$} & X & X & $\checkmark$ &  $\checkmark$ & /\\\cline{2-7}
			\textbf{$\mu_c \neq \mu_t$} & $\Sigma_c \succeq \Sigma_t + k\mathbf{dd}^\top, k > 0$ & X & X & $\checkmark$ & X & /\\\cline{2-7}
			&\textbf{$\Sigma_c \succeq \Sigma_t$} & X & X & X & X & /\\\hline
			&\textbf{$\Sigma_c \nsucceq \Sigma_t$} & X & X & X & X & /
		\end{tabular}} 
		\caption{This table summarizes when different definitions of conservative will be considered true for Gaussian distributions.  If there is a $\checkmark$, then $p_c=\mathrm{N}(\mu_c,\Sigma_c)$ is conservative w.r.t. $p_t=\mathrm{N}(\mu_t, \Sigma_t)$, an ``X" means it is never true, and a ``/'' means it is sometimes true.  To make this table exact, each row is assumed to not include the conditions covered by the row above it.  For example, the row $\Sigma_c \succeq \Sigma_t$ does not include the cases when $\Sigma_c \succ \Sigma_t$.}  
		\label{tab:gaussCompare}
	\end{centering}
\end{table*}

Table~\ref{tab:gaussCompare} has a couple of interesting results. First, note that when $\mu_c=\mu_t$, SC, WC and p.s.d. conservative all consider the same sets of PDFs to be conservative.  If $\mu_c\neq\mu_t$, the relationship between p.s.d. and the WC definition is a bit more nuanced.  Note that the p.s.d. definition really pertains to Question \#2, requiring the covariance matrix to be larger (in the p.s.d. sense) \emph{than the underlying random variable's second moment}.  If $N(\mu_t,\Sigma_t)$ is the optimal PDF, then to be p.s.d. conservative,
\begin{equation}
    \Sigma_c \succeq \Sigma_t + (\mu_c-\mu_t)(\mu_c-\mu_t)^\top
    \label{eq:psdVariance}
\end{equation}
assuming $\mu_c$ and $\mu_t$ are column vectors.  WC, however, only guarantees that $\Sigma_c \succeq \Sigma_t$ \emph{and} that the effects of the different means are overcome before $\alpha=1$.  Note that in Table~\ref{tab:gaussCompare}, anything that is p.s.d. conservative is also WC, but WC has weaker requirements for how much the covariance will increase given a difference between $\mu_c$ and $\mu_t$.

\subsection{Non-Gaussian PDFs}
When comparing definitions of conservative for non-Gaussian PDFs, the p.s.d. definition cannot be included.  We also have the following properties relating different definitions:
\begin{itemize}
	\item Any SC distribution is also WC.
	\item The maximum likelihood modes have to be the same between two distributions for both the GEOP and SC distributions.
	\item The GEKL distribution does not correspond well with the intuitive understanding of conservativeness inherent in SC or p.s.d. definitions.
\end{itemize}
Therefore, we focus on comparing the GEOP and SC definitions of conservative.  By example, we prove that the sets of GEOP and SC functions overlap (Example~\ref{ex: Exp_GEOP_SC_Equiv}), but that neither is a strict subset of the other (Examples \ref{ex:SCnotGEOP1}-\ref{ex: GEOP_Not_SC}).  

\begin{example}
	Let $p_c = \text{Exponential}(\lambda_c)$ and $p_t = \text{Exponential}(\lambda_t)$. The PDF $p_c$ is a strictly conservative w.r.t $p_t$ if and only if $p_c$ is GEOP conservative w.r.t $p_t$.
	\label{ex: Exp_GEOP_SC_Equiv}
\end{example}
\begin{proof}
	To begin, assume $p_c$ is a GEOP approximation of $p_t$. Recall that the entropy of Exponential$(\lambda)$ is $1 - \log(\lambda)$. Since $H(p_c) \geq H(p_t)$, it follows that $\lambda_c \leq \lambda_t$. Let $\alpha \in [0, 1)$. We construct a MV set of the form $[0, x_c]$ with area $\alpha$ for $p_c$. A similar construction holds for MV sets of the form $[0, x_t]$ for $p_t$. The value of $x_c$ must satisfy $\alpha = \int_0^{x_c} p_c(y) \ dy$, so
	\begin{equation*}
	x_c = -\frac{\log(1 -\alpha)}{\lambda_c} \;.
	\end{equation*}
	Since $\lambda_c \leq \lambda_t$, we see that $x_c \geq x_t$, so $M_c(\alpha) \supseteq M_t(\alpha)$. The case when $\alpha = 1$ is easy because $M_c(\alpha) = M_t(\alpha) = [0, \infty)$. Therefore, $M_c(\alpha) \supseteq M_t(\alpha)$ for all $\alpha \in [0,1]$, so $p_c$ is a strictly conservative approximation of $p_t$. 
	
	Now assume that $p_c$ is a strictly conservative approximation of $p_t$. Since the exponential distribution is monotonically decreasing, the order-preservation property holds trivially. All that remains is to show that $H(p_c) \geq H(p_t).$ Using the above calculations for $x_c$ and $x_t$, it must be that $\lambda_c \leq \lambda_t$. If not, then $M_c(\alpha) \not \supseteq M_t(\alpha)$ for some $\alpha$, which is a contradiction. Then, $1 - \log(\lambda_c) \geq 1 - \log(\lambda_t)$, so $H(p_c) \geq H(p_t).$ We conclude then that $p_c$ is a GEOP conservative approximation of $p_t.$
\end{proof}

\begin{example}
	Let $p_t = \mathrm{U}(a, b)$ and $p_c = \mathrm{U}(c, d)$. If $(c,d) \supset (a,b)$, then $p_c$ is strictly, but not GEOP conservative, w.r.t. $p_t$.
	\label{ex:SCnotGEOP1}
\end{example}
\begin{proof}
	Let $\alpha \in (0, 1).$ Since $p_c$ and $p_t$ are ``flat'', the MV sets are not unique. In this proof, we will just show that there exists MV sets that satisfy the requirement $M_c(\alpha) \supseteq M_t(\alpha).$ It is easy to see that one such example of a MV set for $p_c$ with area $\alpha$ is
	\begin{equation*}
	M_c(\alpha) =  \left( \frac{a + b}{2} - \frac{\alpha}{2(b-a)}, \frac{a+b}{2} + \frac{\alpha}{2(b-a)} \right) \;.
	\end{equation*} 
	In a similar way, one MV set for $p_t$ with area $\alpha$ is
	\begin{equation*}
	M_t(\alpha) = \left( \frac{a+b}{2} - \frac{\alpha}{2(d-c)}, \frac{a+b}{2} + \frac{\alpha}{2(d-c)} \right) \;.
	\end{equation*}
	Note that these sets are centered at the mean of $p_t$, $\frac{a+b}{2}.$ Since $(c, d) \supset (a, b)$, it follows that $d - c \geq b - a$, so $M_c(\alpha) \supseteq M_t(\alpha).$ To prove that $p_c$ is not a GEOP conservative approximation of $p_t$, choose two points, $x_1\in(a,b)$ and $x_2\in(c,d) \cap (a,b)^c$.  In this case $p_c(x_2) \geq p_c(x_1)$, which to be GEOP conservative requires $p_t(x_2) \geq p_t(x_1)$.  Because it does not, this is not GEOP conservative.
\end{proof}

\begin{example}
	Let $p_c = \mathrm{N}(0,3)$ and 
	$$p_t(x) = \begin{cases}
	k_1\mathrm{N}(x;0,1) &: x \leq 0 \\
	k_2\mathrm{N}(x;0,2) &: x \geq 0
	\end{cases}$$
	where $k_1$ and $k_2$ are chosen so the integral of $p_t$ is 1 and $p_t$ is continuous ($k_1 \approx .83$, $k_2 \approx 1.17$).  Because $p_t$ is non-symmetric and $p_c$ is symmetric, $p_c$ cannot be GEOP conservative w.r.t. $p_t$, but is strictly conservative.
	\label{ex:SCnotGEOP2}
\end{example}

While there are many such examples where a function can be SC, but not GEOP, being GEOP is not a sufficient condition for being SC as shown in the following example.
\begin{example}
	Let $p_t = \text{Exponential}(1)$ and
	$$p_c(x) = \begin{cases}
	e^{-\lambda x}, & 0 \leq x \leq 1\\
	k e^{-\lambda x}, & 1 < x
	\end{cases}$$
	where $\lambda < 1$ and $k$ is such that $p_c$ integrates to 1\footnote{Specifically, $k=e^\lambda(\lambda-1) + 1$}. We plot $p_c$ in  Figure~\ref{fig:GEOPnotSC} with $\lambda = 0.8$. Note that $p_c$ and $p_t$ are decreasing, so $p_c$ is order preserving w.r.t. $p_t$. When $\lambda = 0.8$, $H(p_t) = 1$ and $H(p_c) \approx 1.48$. So $p_c$ is GEOP conservative w.r.t. $p_t.$ However, $p_c$ is not strictly conservative w.r.t. $p_t$ because for small $\alpha, \ M_t(\alpha) \not \subseteq M_c(\alpha).$
	\label{ex: GEOP_Not_SC}
\end{example}

As these examples show, neither GEOP nor SC conservative is a subset of the other definition.

\begin{figure}[t]
	\begin{subfigure}{.47\textwidth}
	\centerline{\includegraphics[width=.95\linewidth]{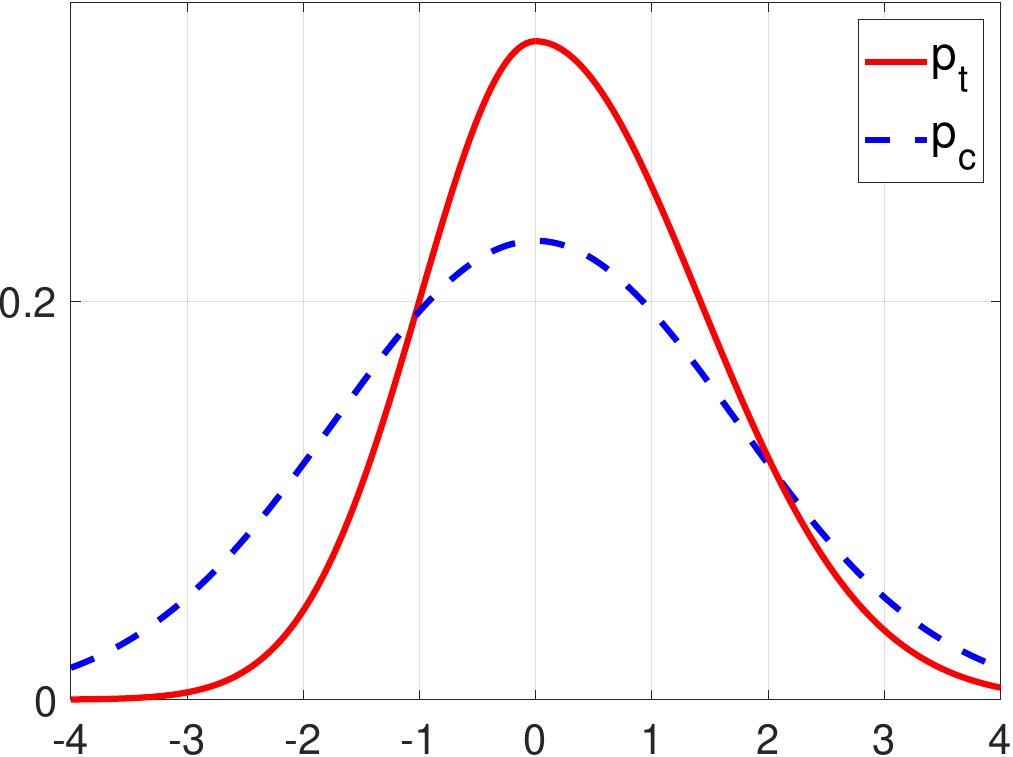}}
	\caption{$p_t$ and $p_c$ from Example \ref{ex:SCnotGEOP2}.  $p_c$ is SC, but not GEOP, w.r.t $p_t$.}
	\end{subfigure}
	\hspace{.05\textwidth}
	\begin{subfigure}{.48\textwidth}
	\centerline{\includegraphics[width=.95\linewidth]{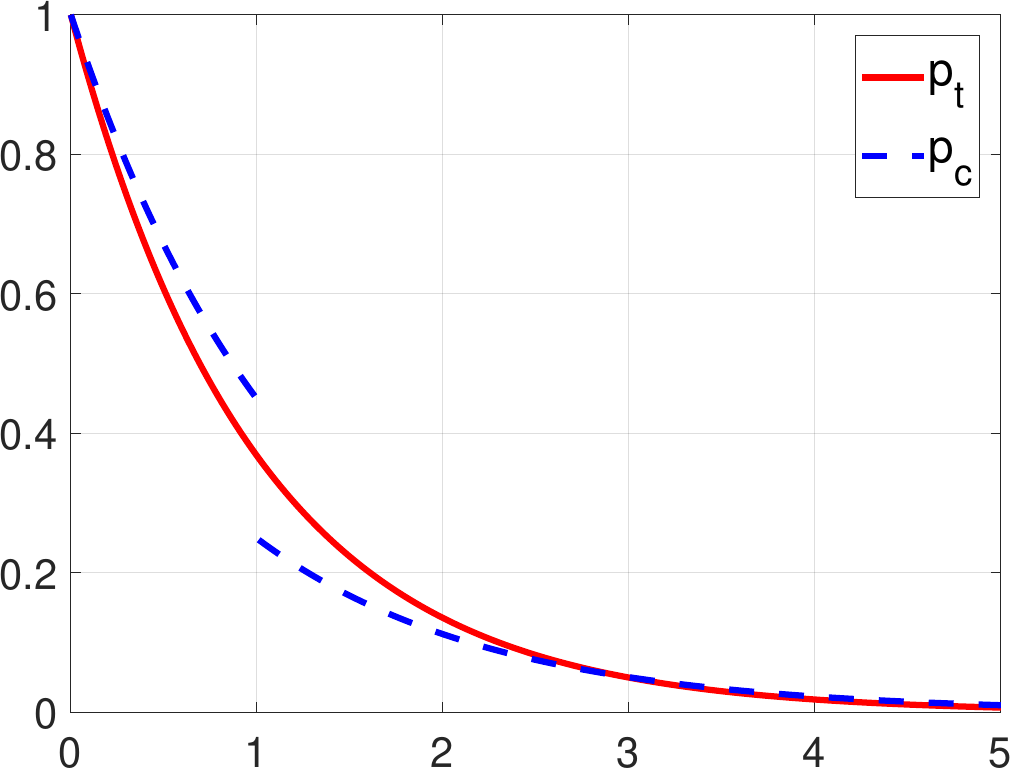}}
	\caption{$p_t$ and $p_c$ from Example \ref{ex: GEOP_Not_SC} with $\lambda = 0.8$. $p_c$ is GEOP, but not SC, w.r.t $p_t.$}
	\end{subfigure}
	\caption{Example PDFs demonstrating the disjoint nature of the GEOP and SC definitions.  Subfigure (a) is SC but not GEOP conservative, while (b) is GEOP but not SC.}
	\label{fig:GEOPnotSC}
\end{figure}

\section{Applications of Conservative Definitions}
\label{s:Applications}
In this section, we demonstrate the utility of having formal definitions of conservativeness.  In particular, we analyze two fusion rules that have seen significant interest in the research literature lately (e.g.~\cite{li2019second,campbell2016distributed,Julier,ahmed2012fast,hurley2002information,bishop2014information,da2021recent}): the log-linear (also called Chernoff or geometric average) fusion and linear (also called arithmetic average) fusion.  While both of these techniques have been proven to be conservative in the Gaussian case, there is also interest in extending these techniques to the non-Gaussian case where there has not previously been any formal proof of their conservativeness.  Using the definition of weakly conservative introduced earlier, we are able to prove the conservativeness of these commonly used data fusion techniques for non-Gaussian PDFs that have infinite support on $\mathbb{R}^m$.

Further demonstration of the utility of the weakly conservative definition is given by proving that a generalization of the linear and log-linear averaging techiques -- power mean fusion -- can also be shown to be conservative.  Finally, we also prove that Bayesian fusion of conservative PDFs also yields conservative PDFs.  

This section is organized as follows.  In the first sub-section, we introduce two properties of weakly conservative that can be used to prove conservativeness.  In the second sub-section, we prove that the linear, log-linear, and power mean fusion rules generate conservative PDFs.  The third sub-section proves that when using a Bayesian update to combine two PDFs, if one of the input PDFs is the output of a (previously discussed) conservative fusion rule, then the output is also a conservative approximation.  Note that proving these properties has not been performed previously for non-Gaussian distributions and represent novel contributions of this paper in their own right.

\subsection{Useful Properties of Weakly Conservative}
The following propositions are often useful to determine if a PDF is weakly conservative.  
\begin{proposition}
	Let $p_c$ and $p_t$ be PDFs that satisfy Condition \ref{cond:SupportSuperset}. Let  $A = \{x: p_c(x) < p_t(x) \}$ and $\varepsilon = \inf_{x \in A} p_c(x)$. If $P_c(\{x: \ p_c(x)<\varepsilon\}) > 0$ and $P_t(\{x: p_t(x) < \varepsilon\}) > 0$ then $p_c$ is weakly conservative w.r.t. $p_t$. 
	\label{prop: AlphaCondSuff}
\end{proposition}
\begin{proof}
	We prove that Condition \ref{cond:Broader_Cons3} holds for all $\alpha$ in some interval. We define $\alpha'$ by
	\begin{equation*}
	\alpha' = \int_{S_c(\varepsilon)}p_c(x) \ dx \;.
	\end{equation*}
	That $\alpha' < 1$ follows because $P_c(\{x: p_c(x)<\varepsilon\}) > 0$.  For any $\alpha \in [\alpha',1)$, $M_c(\alpha) \supseteq A$.  To show this, note that $M_c(\alpha') = \{x: p_c(x) \geq \varepsilon\}$ by the definition of $\varepsilon$ and Property \ref{prop: LevelSet_MV}. If $x \in A$, it follows that $p_c(x) \geq \epsilon$, so $x \in M_c(\alpha')$ and $M_c(\alpha') \supseteq A.$ 
	
	We use Property \ref{prop: Increasing_MV} to conclude that for $\alpha \in [\alpha', 1), \ M_c(\alpha) \supseteq A.$ To continue, note that for $x \in M^\mathsf{c}_c(\alpha), \ p_c(x) \geq p_t(x)$
	because $M^c_c(\alpha) \subseteq A^c.$ Thus $P_c(M^\mathsf{c}_c(\alpha)) \geq P_t(M^\mathsf{c}_c(\alpha))$. Since $P_t(M^\mathsf{c}_c(\alpha)) + P_t(M_c(\alpha)) = P_c(M^\mathsf{c}_c(\alpha)) + P_c(M_c(\alpha))$, we conclude that $P_t(M_c(\alpha)) \geq P_c(M_c(\alpha)).$ This proves Condition \ref{cond:Broader_Cons3} for all $ \alpha\in[\alpha',1)$. 
	With a similar argument, we can can conclude that Condition \ref{cond:Broader_Cons2} holds for all $\alpha \in [\alpha'', 1)$ where $\alpha'' < 1$. Because both $\alpha'$ and $\alpha''$ are less than one, $p_c$ is a weakly conservative approximation of $p_t$.
\end{proof}

When working with PDFs with infinite support, the following Proposition is even more straightforward.
\begin{proposition}
	\label{prop:boundedA}
	Let $p_c$ and $p_t$ be PDFs with support $\mathbb{R}^m$. Let $A = \{x: p_c(x) < p_t(x)\}$. If $A$ is bounded, then $p_c$ is weakly conservative w.r.t $p_t.$
	\begin{proof}
		Let $A = \{x: p_c(x) < p_t(x)\}$ and let $\epsilon = \inf_{x \in A}p_c(x).$ Note that $\epsilon > 0$ because the support of $p_c$ is $\mathbb{R}^m$. Because $p_c$ and $p_t$ go to zero as $||x||^2 \rightarrow \infty$,
		the set $\{x: p_c(x) < \epsilon\}$ is unbounded. Because the support of each PDF is $\mathbb{R}^m$,  it follows that $P_c(\{x: p_c(x) < \epsilon\}) > 0$ and $P_t(\{x: p_t(x) < \epsilon\}) > 0$. We then use  Proposition \ref{prop: AlphaCondSuff} to conclude that $p_c$ is weakly conservative w.r.t $p_t.$
	\end{proof}
\end{proposition}

\subsection{Proving a Fusion Rule is Conservative}
\label{ss:FusionRules}
In this section, we prove that three previously used fusion rules produce weakly conservative PDFs w.r.t. the ideally (perfect knowledge) fused PDF under some limiting assumptions.  These assumptions include:
\begin{itemize}
    \item Each PDF has infinite support: $supp(p) = \mathbb{R}^m$.
    \item Each fusion rule takes in multiple (marginalized) PDFs, denoted $\{p_i\}_{i = 1}^n$
    \item Each  $p_i$ can be factored as $p_i(x) \propto p_C(x) p_{i \setminus C}(x)$, where $p_C(x)$ is the common information and both $p_C(x)$ and $p_{i \setminus C}(x)$ can be normalized to integrate to one across all $\mathbb{R}^m$.
    \item The ``true'' distribution (assuming the common information was perfectly known) can be represented as:
        \begin{equation}
        p_{t}(x) = \frac{1}{\eta_t}p_C(x)\prod_{i = 1}^n p_{i\setminus C}(x)\ dx  \;,
        \label{eq: DEF_TRUE}
        \end{equation}
        where $\eta_t$ is the normalizing constant. 
\end{itemize} 

Note that these assumptions apply to all proofs in this section.

\subsubsection{The Linear Opinion Pool}
The first fusion rule we study is the linear opinion pool (LOP), \cite{clemen1999combining} and \cite{Abbas2009}. The LOP method forms a convex combination $p_f$, given by
\begin{equation}
p_f(x) = \sum_{i = 1}^n \omega_i p_i(x), \ \ \ \text{with } \sum_{i = 1}^n \omega_i = 1 \;,
\label{eq: DEF_LOP}
\end{equation}
and $\omega_i \geq 0$.

\begin{proposition}
The PDF $p_f$ created by the LOP method in (\ref{eq: DEF_LOP}) is weakly conservative w.r.t. $p_{t}$ from (\ref{eq: DEF_TRUE}). 
\label{prop: Linear_Cons}
\end{proposition}

\begin{proof}
We first show that there exists an $x$ such that $p_f(x) \geq p_t(x).$ To show this, note
\begin{align*}
h(x) \overset{\triangle}{=} & \frac{p_f(x)}{p_t(x)} = \eta_t\sum_{i = 1}^n \frac{\omega_i}{\prod_{j \neq i} p_{j\setminus C}(x)} 
\end{align*}
diverges as $||x||^2 \rightarrow \infty$ because $\lim_{||x||^2 \rightarrow \infty} p_{i\setminus C}(x) = 0$ for $1 \leq i \leq n$. Since $h$ diverges, there exists a finite $a$ such that $h(x) > 1$ for $||x||^2 > a$, i.e., $p_f(x) > p_t(x)$. Because $a$ is finite, the set $A=\{x : p_f(x) < p_t(x)\}$ is bounded and $p_f$ is weakly conservative w.r.t. $p_c$ by Proposition~\ref{prop:boundedA}.
\end{proof}

\subsubsection{The Log Linear Opinion Pool} The second fusion rule we are interested in is the Chernoff fusion method, sometimes known as the log-linear opinion pool (LLOP) \cite{Julier}. Given PDFs $(p_i)_{i = 1}^n$, we fuse them together to produce a PDF $p_f$
\begin{equation}
p_f(x) = \frac{1}{\eta_f}\prod_{i = 1}^n p_{i}^{\omega_i}(x)\;, 
\label{eq: DEF_LLOP}
\end{equation}
where $$\sum_{i=1}^n\omega_i =1,\; 0\leq \omega_i \leq 1$$ and $$\eta_f = \int_{\mathbb{R}^m}\prod_{i = 1}^n p_i^{\omega_i}(x) \ dx\; .$$

\begin{proposition}
The PDF $p_f$ created by the LLOP method in (\ref{eq: DEF_LLOP}) is a weakly conservative approximation of $p_{t}$ in (\ref{eq: DEF_TRUE}).
\label{prop: Log_Linear_Cons}
\end{proposition}

\begin{proof}
We first show that there exists an $x$ such that $p_f(x) \geq p_t(x).$ To show this, we define the ratio between $p_f$ and $p_t$
\begin{align*}
h(x) \overset{\triangle}{=} & \frac{p_f(x)}{p_t(x)} = \frac{\eta_t}{\eta_f}\prod_{i = 1}^n p_{i \setminus C}^{\omega_i-1}(x)
\end{align*}
and note that it diverges as $||x||^2 \rightarrow \infty$ since $\omega_i \leq 1$ for all $i$. The rest of the proof is identical to the proof of Proposition \ref{prop: Linear_Cons}.
\end{proof}

\begin{figure*}
	\begin{subfigure}{0.31\linewidth}
		\centering
		\includegraphics[width=1\linewidth]{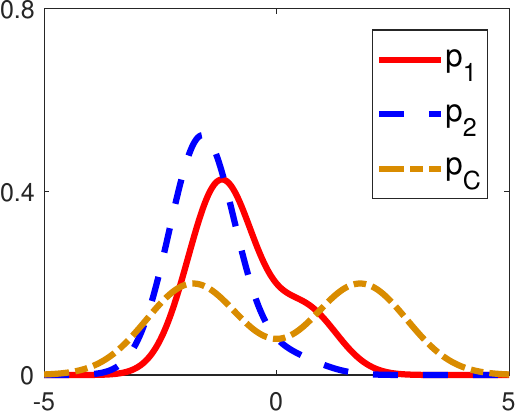}
	\end{subfigure}
	\hspace{.01\linewidth}
\begin{subfigure}{.31\linewidth}
	\centering
	\includegraphics[width=1\linewidth]{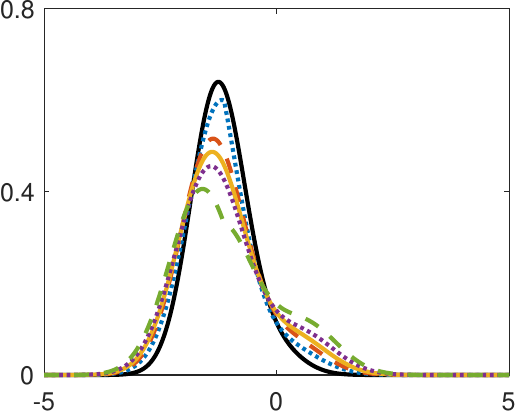}
\end{subfigure}
	\hspace{.01\linewidth}
\begin{subfigure}{.31\linewidth}
	\centering
	\includegraphics[width=1\linewidth]{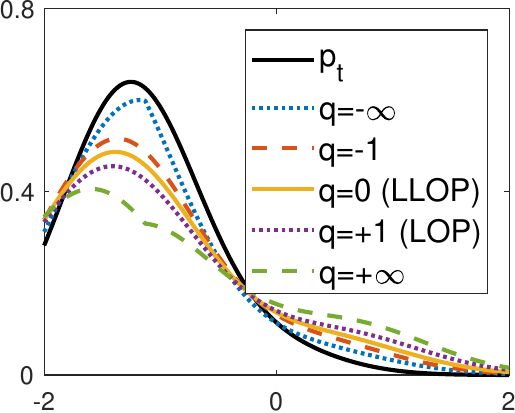}
\end{subfigure}
	\caption{A simple example of homogeneous fusion, with $p_C = 0.5\mathrm{N}(-1.8,1) + 0.5\mathrm{N}(1.8,1)$, $p_{1\setminus C} = \mathrm{N}(-0.6,1)$ and $p_{2\setminus C} = \mathrm{N}(-1.4, 1)$.  The left figure plots $p_C$ and the two input functions $p_1$ and $p_2$.  The middle figure plots the ``true'' fusion result $p_t = \frac{1}{\eta_t}p_C\, p_{1\setminus C}\, p_{2\setminus C}$, together with a homogeneous fusion with varying values of $q$.  The third figure is a zoomed in version of the middle sub-figure.  Where applicable, the weight on each input was set to 0.5}
	\label{fig:ExampleHomogeneous}
\end{figure*}

\subsubsection{Homogeneous Functionals} 
The third data fusion rule we analyze fuses PDFs using homogeneous functionals from \cite{Taylor2019}. The ``generalized power mean'' can be used to create several different homogeneous functions of degree 1.  Given PDFs $\{p_i\}_{i = 1}^m$, the fused PDF $p_f$ is defined as
\begin{equation}
p_f(x) = \frac{1}{\eta_f}\left(\sum_{i = 1}^n \omega_i p^q_{i}(x)\right)^{1/q} 
\label{eq: Fused_Homogeneous}
\end{equation}
for $-\infty \leq q \leq \infty$. Various data fusion methods are special cases of the ``generalized power mean rule'' method. For example,
if $q = 0$, we recover the LLOP method and if $q = 1$ we recover the LOP method. In addition, if $q = -\infty$, $p_f(x) \propto \min_{1 \leq i \leq n}(p_i(x))$ and if $q = \infty$, $p_f \propto \max_{1 \leq i \leq n}(p_i(x))$. In Figure \ref{fig:ExampleHomogeneous}, we plot $p_f$ for various values of $q$, demonstrating the different results that can be obtained using homogeneous functionals.

\begin{proposition}
The PDF $p_f$ created by the fusion method in (\ref{eq: Fused_Homogeneous}) is a weakly-conservative approximation of $p_{t}$ from (\ref{eq: DEF_TRUE}). 
\end{proposition}
\begin{proof}
	By the definition of homogeneous functionals of degree 1, \eqref{eq: Fused_Homogeneous} can be re-written as:
	$$p_f(x) = \frac{1}{\eta_f}p_C(x)\left(\sum_{i = 1}^n \omega_i p^q_{i \setminus C}(x)\right)^{1/q}$$
We define 
\begin{align*}
h(x) &\overset{\triangle}{=} \frac{p_f(x)}{p_t(x)} = \frac{\eta_t}{\eta_f}\left(\sum_{i = 1}^n \frac{\omega_i}{\prod_{j \neq i} p^q_{j \setminus C}(x)}\right)^{1/q} \;.
\end{align*}
For any $q \in [-\infty, \infty]$, $h(x) \rightarrow \infty$ as $||x||^2 \rightarrow \infty$. The rest of the proof is identical to the proof of Proposition \ref{prop: Linear_Cons}.
\end{proof}

\section{Conclusion}
\label{sec: Conclusion}
When working with intractable PDFs it is often desirable to have a supplementary PDF that has some properties w.r.t. the true PDF.  While the idea of ``conservativeness'' has been mentioned previously as a desirable characteristic, there is little consensus on a general definition of conservativeness. This paper introduces an intuitive and formal definition for conservative that can be applied to any two PDFs.  This definition, \emph{strictly conservative}, captures the intuition behind conservativeness being a desirable property and conforms fairly well with prior definitions, while addressing their shortcomings.  Unfortunately, we show that no current fusion algorithm achieves strictly conservative outputs.  Therefore, we also propose a weaker definition. Using this weaker definition, we prove that several previously introduced fusion rules are ``weakly'' conservative for PDFs with some significant restrictions (infinite support and integrates to 1 over $\mathbb{R}^m$). 

While these properties are useful, there is considerable future work that we would like to see performed in this area.  First, the definitions of conservative presented in this paper are all binary. The $\alpha'$ parameter used to prove weakly conservative can be arbitrarily close to 1.  This $\alpha'$ parameter, however, is also a measure of how similar two PDFs are.  Designing fusion rules that guarantee maximum $\alpha'$ values could be very meaningful.  Second, all of the proofs in Section \ref{s:Applications} are for PDFs with infinite support and normalizable across that support.  Extending these proofs to (1) allow the independent information ($p_{i\bs C}$) to be likelihoods rather than proper PDFs and (2) including all PDFs with non-infinite support could significantly expand the applicability of these results.  Third, it would be interesting to consider other desirable attributes of conservative distributions and their definitions such as unbiasedness or invariance to coordinate transformation and see if they could be useful in defining conservativeness.

\bibliographystyle{plain}
\bibliography{ConsistencyThoughts}

\begin{thebibliography}{10}

\bibitem{Abbas2009}
Ali~E. Abbas.
\newblock A kullback-leibler view of linear and log-linear pools.
\newblock {\em Decision Analysis}, February 2009.

\bibitem{ahmed2012fast}
Nisar~R Ahmed and Mark Campbell.
\newblock Fast consistent chernoff fusion of gaussian mixtures for ad hoc
  sensor networks.
\newblock {\em IEEE transactions on signal processing}, 60(12):6739--6745,
  2012.

\bibitem{Ajgl}
J.~Ajgl and M.~{\v S}imandl.
\newblock Conservative merging of hypotheses given by probability densities.
\newblock {\em 2012 15th International Conference on Information Fusion}, pages
  1884--1890, July 2012.

\bibitem{AjglSimandl}
J.~Ajgl and M.~{\v S}imandl.
\newblock On conservativeness of posterior density fusion.
\newblock {\em Proceedings of the 16th International Conference on Information
  Fusion}, pages 85--92, July 2013.

\bibitem{ajgl2014conservativeness}
Ji{\v{r}}{\'\i} Ajgl and Miroslav {\v{S}}imandl.
\newblock Conservativeness of estimates given by probability density functions:
  Formulation and aspects.
\newblock {\em Information Fusion}, 20:117--128, 2014.

\bibitem{bailey2012conservative}
Tim Bailey, Simon Julier, and Gabriel Agamennoni.
\newblock On conservative fusion of information with unknown non-gaussian
  dependence.
\newblock In {\em 2012 15th International Conference on Information Fusion},
  pages 1876--1883. IEEE, 2012.

\bibitem{bishop2014information}
Adrian~N Bishop.
\newblock Information fusion via the wasserstein barycenter in the space of
  probability measures: Direct fusion of empirical measures and gaussian fusion
  with unknown correlation.
\newblock In {\em 17th International Conference on Information Fusion
  (FUSION)}, pages 1--7. IEEE, 2014.

\bibitem{campbell2016distributed}
Mark~E Campbell and Nisar~R Ahmed.
\newblock Distributed data fusion: Neighbors, rumors, and the art of collective
  knowledge.
\newblock {\em IEEE Control Systems Magazine}, 36(4):83--109, 2016.

\bibitem{clemen1999combining}
R.T. Clemen and R.L. Winkler.
\newblock Combining probability distributions from experts in risk analysis.
\newblock {\em Risk Analysis}, 19(2):187--203, 1999.

\bibitem{da2021recent}
Kai Da, Tiancheng Li, Yongfeng Zhu, Hongqi Fan, and Qiang Fu.
\newblock Recent advances in multisensor multitarget tracking using random
  finite set.
\newblock {\em Frontiers of Information Technology \& Electronic Engineering},
  22(1):5--24, 2021.

\bibitem{garcia2003level}
Javier~Nu{\~n}ez Garcia, Zoltan Kutalik, Kwang-Hyun Cho, and Olaf Wolkenhauer.
\newblock Level sets and minimum volume sets of probability density functions.
\newblock {\em International journal of approximate reasoning}, 34(1):25--47,
  2003.

\bibitem{hurley2002information}
Michael~B Hurley.
\newblock An information theoretic justification for covariance intersection
  and its generalization.
\newblock In {\em Proceedings of the Fifth International Conference on
  Information Fusion. FUSION 2002.(IEEE Cat. No. 02EX5997)}, volume~1, pages
  505--511. IEEE, 2002.

\bibitem{Julier}
S.~J. Julier.
\newblock An empirical study into the use of chernoff information for robust,
  distributed fusion of gaussian mixture models.
\newblock {\em 2006 9th International Conference on Information Fusion}, pages
  1--8, July 2006.

\bibitem{julier2006using}
Simon~J Julier, Tim Bailey, and Jeffrey~K Uhlmann.
\newblock Using exponential mixture models for suboptimal distributed data
  fusion.
\newblock In {\em Nonlinear Statistical Signal Processing Workshop, 2006 IEEE},
  pages 160--163. IEEE, 2006.

\bibitem{julier1997non}
S.J. Julier and J.K. Uhlmann.
\newblock A non-divergent estimation algorithm in the presence of unknown
  correlations.
\newblock In {\em Proceedings of the 1997 American Control Conference}, pages
  2369--2373, 1997.

\bibitem{li2019second}
Tiancheng Li, Hongqi Fan, Jes{\'u}s Garc{\'\i}a, and Juan~M Corchado.
\newblock Second-order statistics analysis and comparison between arithmetic
  and geometric average fusion: Application to multi-sensor target tracking.
\newblock {\em Information Fusion}, 51:233--243, 2019.

\bibitem{mutambara1998decentralized}
Arthur~GO Mutambara.
\newblock {\em Decentralized estimation and control for multisensor systems}.
\newblock CRC press, 1998.

\bibitem{Niehsen}
W.~Niehsen.
\newblock Information fusion based on fast covariance intersection filtering.
\newblock {\em Proceedings of the Fifth International Conference on Information
  Fusion. FUSION 2002.}, 2:901--904, July 2002.

\bibitem{Noack}
B.~Noack, M.~Baum, and U.~D. Hanebeck.
\newblock Covariance intersection in nonlinear estimation based on pseudo
  gaussian densities.
\newblock {\em 14th International Conference on Information Fusion}, pages
  1--8, July 2011.

\bibitem{noack2017decentralized}
Benjamin Noack, Joris Sijs, Marc Reinhardt, and Uwe~D Hanebeck.
\newblock Decentralized data fusion with inverse covariance intersection.
\newblock {\em Automatica}, 79:35--41, 2017.

\bibitem{Taylor2019}
Clark~N. Taylor and Adrian~N. Bishop.
\newblock Homogeneous functionals and bayesian data fusion with unknown
  correlation.
\newblock {\em Information Fusion}, 45:179--189, January 2019.

\bibitem{wang2012distributed}
Yimin Wang and X~Rong Li.
\newblock Distributed estimation fusion with unavailable cross-correlation.
\newblock {\em IEEE Transactions on Aerospace and Electronic Systems},
  48(1):259--278, 2012.

\end{thebibliography}

\appendix
\section{Proof of Proposition \ref{prop: ThreeCond_StrictCons}}
\begin{proof}
First, we prove that being strictly conservative implies each of the three conditions. We start with Condition \textbf{\ref{cond:SupportSuperset}.}  Assume for contradiction that $\text{supp}(p_c) \nsupseteq \text{supp}(p_t)$.  Then $M_c(1) \nsupseteq M_t(1)$, violating Definition~\ref{def: Original_Cons}. We now analyze Condition \textbf{\ref{cond:Broader_Cons2}.} When $M_c(\alpha) \supseteq M_t(\alpha)$, we can re-write $M_c(\alpha) = M_t(\alpha)\cup B$ where $B=M_c(\alpha)\cap M_t(\alpha)^c$.  If $P_c(M_c(\alpha)) = P_t(M_t(\alpha)) = \alpha$, $P_c(M_t(\alpha))+P_c(B) = P_c(M_c(\alpha))$ or $P_c(M_t(\alpha)) = P_t(M_t(\alpha)) - P_c(B)$.  Because $P_c(B) \geq 0$, Condition~\ref{cond:Broader_Cons2} is proven.  For Condition~\ref{cond:Broader_Cons3}, the proof is similar to the proof for Condition~\ref{cond:Broader_Cons3} and is omitted for brevity.

Second, we prove by counter-example that these three conditions are not sufficient for $P_c$ to be strictly conservative approximation of $P_t$. Let $p_t = \mathcal{N}(0,4)$ and 
$$
p_c(x) = \begin{cases}
.05 & 1 \le x \le 2\\
\frac{\Phi(\frac{15}{2}) - \Phi(\frac{10}{2})+\Phi(\frac{2}{2}) - \Phi(\frac{1}{2}) - .05}{5} & 10 \le x \le 15\\
\mathcal{N}(x; 0,4) &\text{otherwise}.
\end{cases}
$$
\begin{figure}
	\centering
	\includegraphics[width=.7\linewidth]{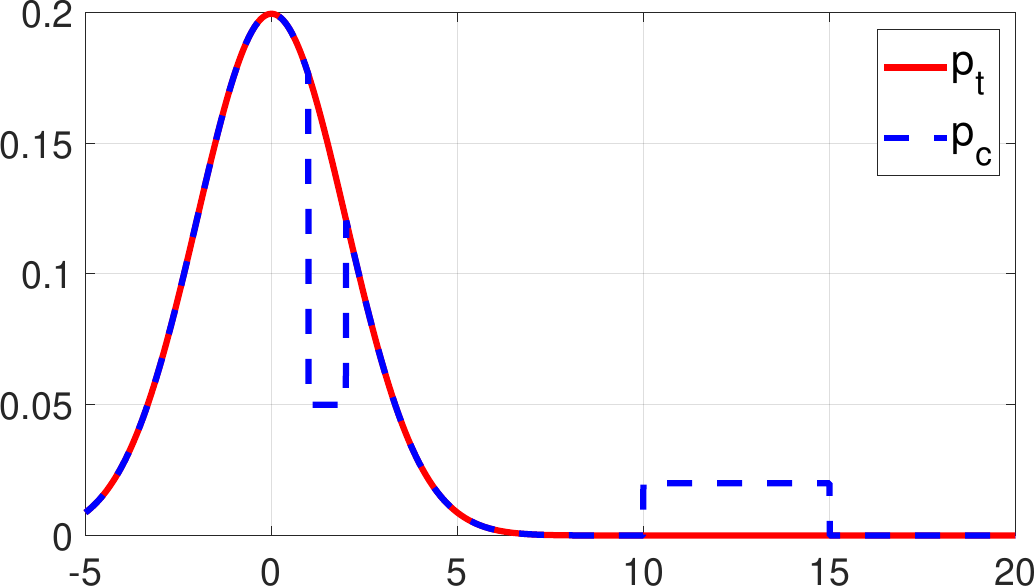}
	\caption{An example where $p_c$ obeys the three conditions across all $\alpha$ but is not strictly conservative.}
	\label{fig:CounterExampleStrictly}
\end{figure}
These two distributions are illustrated in Figure~\ref{fig:CounterExampleStrictly}.  While $p_c$ meets all three conditions for all $\alpha$, $p_c$ will not be strictly conservative w.r.t. $p_t$ due to the ``notch'' removed from 1 to 2, providing the needed counter-example.
\end{proof}

\section{Proof for some entries in Table \ref{tab:gaussCompare}}
\label{a:tableProofs}
The following propositions and their proofs help define why different row and column combinations in Table~\ref{tab:gaussCompare} have a checkmark or X.  All distributions ($p_c$ and $p_t$) are assumed to be Gaussian distributions.
\begin{proposition}
If $\mu_c = \mu_t$, then $p_c$ is p.s.d. conservative w.r.t. $p_t$ if and only if $p_c$ is strictly conservative w.r.t. $p_t$. 
\end{proposition}

\begin{proof}
    Assume that $p_c$ is a p.s.d. conservative approximation of $p_t$. Without loss of generality, assume that $\mu_c = \mu_t = 0.$ The MV set for $p_c$ and $p_t$ are $M_c(\alpha) = \{x: x^T \Sigma_c^{-1}x \leq F^{-1}(\alpha)\}$ and $M_t(\alpha) = \{x: x^T\Sigma_t^{-1} x \leq F^{-1}(\alpha)\}$, respectively, where $F^{-1}(\alpha)$ is the inverse CDF of the $\chi^2$ distribution with $\text{dim}(p_c)$ degrees of freedom. Take $x \in M_t(\alpha)$, so that $x^T \Sigma_t^{-1} x \leq F^{-1}(\alpha)$. Then,
    \begin{equation*}
        x^T(\Sigma_t^{-1} - \Sigma_c^{-1})x \geq 0 \implies x^T\Sigma_t^{-1}x \geq x^T \Sigma_c^{-1}x \;.
    \end{equation*}
    We conclude that $x^T \Sigma_c^{-1}x \leq F^{-1}(\alpha)$, so $x \in M_c(\alpha).$ It follows that $M_t(\alpha) \subseteq M_c(\alpha)$ for all $\alpha.$ To prove the other direction, assume for contradiction that $p_c$ is strictly, but not p.s.d conservative, w.r.t $p_t$. Then, there exists $\tilde x$ such that $\tilde x \Sigma_t^{-1}\tilde x < \tilde x \Sigma_c^{-1} \tilde x.$ Let $\alpha = F\left(\tilde x^T\Sigma_t^{-1}\tilde x\right) \in (0,1).$ Then, $\tilde x \in M_t(\alpha)$, but $\tilde x \not \in M_c(\alpha)$ because $\tilde x \Sigma_c^{-1} \tilde x > \tilde x \Sigma_t^{-1} \tilde x = F^{-1}(\alpha).$ This is a contradiction. The result follows.
\end{proof}

\begin{proposition}
    If $\mu_c = \mu_t$ and $\Sigma_c = k\Sigma_t, k\ge 1$ then $p_c$ is GEOP w.r.t. $p_t$
\end{proposition}
\begin{proof}
    First, we prove $p_c$ is order preserving (OP) w.r.t. $p_t$.  Note the following string of inequalities, which hold for any $x_1, x_2$:
    \begin{align*}
       p_c(x_1) \geq p_c(x_2) & \iff \log p_c(x_1) \geq \log p_c(x_2) \\
       &\iff -\frac{1}{2}(x_1-\mu)^T \Sigma_c^{-1}(x_1-\mu) \geq -\frac{1}{2}(x_2-\mu)^T \Sigma_c^{-1}(x_2-\mu) \\
       &\iff -\frac{1}{2k}(x_1 - \mu)^T \Sigma_t^{-1} (x_1-\mu) \geq -\frac{1}{2k}(x_2-\mu)^T \Sigma_t^{-1}(x_2-\mu) \\
       &\iff \log(p_t(x_1)) \geq \log(p_t(x_2)) \\
       &\iff p_t(x_1) \geq p_t(x_2) \;.
    \end{align*}
    
    Second, recall that for $p = \mathrm{N}(\mu, \Sigma)$, the entropy of $p$ is 
    $$H(p) = \frac{1}{2}\log(\text{det}(2\pi e \Sigma)).$$ 
    Therefore, if $k\geq 1$ then $H(p_c) \geq H(p_t)$.
\end{proof}
\end{document}